\documentclass[a4paper,11pt]{article}
 \usepackage[english]{babel}
\usepackage[a4paper]{geometry}
\usepackage{amsmath}
\usepackage{amsthm}
\usepackage{amssymb}
\usepackage{bbm}
\usepackage{color}
\usepackage[dvipsnames]{xcolor}

\usepackage{tikz}
\usetikzlibrary{positioning}
\usetikzlibrary{arrows}
\usepackage{caption}
\usepackage{subcaption}
 
\usepackage{mathrsfs}
\usepackage{enumerate}
 
\usepackage{hyperref} 

\def\bE{\mathbb{E}}
\def\bP{\mathbb{P}}

\def\cF{\mathcal{F}}

\DeclareMathOperator{\Var}{Var}

\def\tbf{\textbf }

\def\be{\begin{equation}}
\def\bEe{\end{equation}}

\newtheorem{theorem}{Theorem}   

\newtheorem{prop}[theorem]{Proposition}
\newtheorem{lemma}[theorem]{Lemma}

\begin{document}
\title{The Order of Free Energy Fluctuations in the Critical Sherrington--Kirkpatrick Model Revisited}
\author{Adrien Schertzer \thanks{Department of mathematics, ENS Lyon and CNRS, Lyon, France} }
\maketitle

\begin{abstract}
We study the fluctuations of the free energy of the Sherrington--Kirkpatrick
model at the critical inverse temperature \(\beta=1\). We prove that
\[
\operatorname{Var}(F_N(1))\leq \frac14\log N+C,
\]
with \(C\) independent of \(N\). This gives a logarithmic upper bound, in
agreement with the order predicted in the physics literature. We also prove
that the critical variance diverges: more precisely,
\[
\operatorname{Var}(F_N(1))
\geq
\frac12\log\log\log N - C .
\]
\end{abstract}


\section{Introduction and main results} 

Consider spin configurations $\sigma=( \sigma_1,\sigma_2,\ldots,\sigma_N) \in \{-1,1\}^N$ whose energies are determined by the Sherrington-Kirkpatrick (SK) \cite{SK} Hamiltonian $H_N:\{-1,1\}^N\to\mathbb{R}$ 
		\begin{equation} \label{eq:HN}
		H_N(\sigma) =    \sum_{1\leq i < j\leq N} g_{ij} \sigma_{i}\sigma_j = \frac{1}{2} (\sigma, \tbf{G}\sigma ).
		\end{equation}
Here \(G\) is a symmetric random matrix with \(g_{ii}=0\), and with \((g_{ij})_{i<j}\) independent centered Gaussian random variables of variance \(N^{-1}\). We assume the $(g_{ij})_{ i<j}$ to be realized in a probability space $(\Omega, \cF, \bP)$ and denote the expectation w.r.t.\ them by $\mathbb{E}(\cdot)$. The free energy $F_N(\beta)$ at inverse temperature $\beta\geq 0$ is defined by
		\be\label{eq:defenergy} F_N(\beta)=  \log Z_N(\beta) \hspace{0.5cm} \text{for} \hspace{0.5cm} Z_N(\beta) = \frac{1}{2^N}\sum_{\sigma \in \{-1,1\}^N}  e^{\beta H_N(\sigma)} \bEe
and by $\langle \cdot \rangle $ we denote from now on the Gibbs measure induced by $H_N$. It is defined so that for every observable $ f: \{-1,1\}^N \to \mathbb{R} $, we have
		\be\label{eq:defGibbs}  \langle f \rangle_{\beta} = \frac{1}{Z_N(\beta)}\frac{1}{2^N} \sum_{\sigma \in \{-1,1\}^N}  f(\sigma) \,e^{\beta H_N(\sigma)}. \bEe
		It was proved by Guerra and Toninelli \cite{GuerraToninelli} that, almost surely,
\[
\lim_{N\to\infty}\frac{F_N(\beta)}{N}
=
\lim_{N\to\infty}\frac{\mathbb E F_N(\beta)}{N}.
\]
This limit is described by the Parisi formula \cite{Par1, Par2, Gue, Tal}. The value \(\beta_c=1\) marks the critical inverse temperature separating the high-temperature  regime from the low-temperature regime. For a thorough introduction to the model, see e.g.\ \cite{Tal1, Tal2, P}. In the high-temperature regime $\beta \leq 1$, the quenched free energy coincides with the annealed free energy in the thermodynamic limit:
\[
\lim_{N\to\infty}\frac{1}{N}\mathbb{E}\log Z_N(\beta)=\lim_{N\to\infty}\frac{1}{N}\log \mathbb{E}Z_N(\beta)=\frac{\beta^2}{4}.
\]
By contrast, in the low-temperature regime, for $\beta>1$, we have 
\[
\lim_{N\to\infty}\frac{1}{N}\mathbb{E}\log Z_N(\beta) < \lim_{N\to\infty}\frac{1}{N}\log \mathbb{E}Z_N(\beta).
\]
Determining the free energy fluctuations of the SK model at the critical inverse temperature \(\beta=1\) is a difficult open problem. 

\begin{itemize}
    \item \textbf{High temperature.}
    Aizenman, Lebowitz, and Ruelle \cite{ALR} proved that, for every fixed
    $\beta<1$, the centered free energy converges in distribution to a
    Gaussian random variable. In particular, $\Var(F_N(\beta))$ converges to a
    finite limit.

    \item \textbf{Near the critical inverse temperature.}
    More recently, Dey and Kang \cite{DeyKang} studied the regime in which the
    critical point is approached from the high-temperature side. In particular,
    for every fixed $c>0$, they proved that
    \[
    \Var\bigl(F_N(1-cN^{-1/3})\bigr)
    =
    \frac16\log N+O(1).
    \]

    \item \textbf{At the critical inverse temperature.}
    At the critical point, Chen and Lam \cite[Theorem~1.1]{CL} proved the upper
    bound
    \[
    \Var(F_N(1)) \leq C\bigl((\log N)^2+1\bigr)
    \]
    for some constant $C>0$. At arbitrary fixed inverse temperature,
    Chatterjee \cite{Cha} established the superconcentration bound
    \[
    \Var(F_N(\beta))\leq C(\beta)\frac{N}{\log N}.
    \]
\end{itemize}

\subsection{Main results}

In this paper, we provide new bounds on the variance of the critical free energy
\(\operatorname{Var}(F_N(1))\). The physics literature predicts the asymptotic
behavior
\begin{equation}\label{eq:conj}
\operatorname{Var}(F_N(1))=\frac16\log N+O(1),
\end{equation}
see \cite{Asp,PR}. This remains open rigorously. Our first result proves the expected logarithmic order as an upper bound,
although with the constant \(1/4\) instead of the conjectural constant \(1/6\). In particular, it improves the previously known \((\log N)^2\) upper bound of Chen and Lam \cite{CL}.

\begin{theorem}[Upper bound]\label{thm:main}
There exists a constant \(C>0\), independent of \(N\), such that
\[
\operatorname{Var}\big(F_N(1)\big)
\leq
\frac14 \log N + C .
\]
\end{theorem}

The proof is based on the second-moment estimate
\[
\operatorname{Var}\big(F_N(\beta)\big)
\leq
\log
\frac{\mathbb E\big[ Z_N(\beta)^2\big]}
{\big(\mathbb E Z_N(\beta)\big)^2}.
\]
We prove this inequality in the next section for the Sherrington--Kirkpatrick
model. The argument uses only convexity and a comparison between the first and
second moments of the partition function, and should therefore apply to a
broader class of Gaussian disordered systems. We also note that the Curie--Weiss comparison is robust enough to extend slightly beyond the critical point, giving near-critical low-temperature bounds
in the same spirit as Chen and Lam. We do not emphasize this extension, since the main focus of the present note is the behavior exactly at criticality. Our second result gives a lower bound showing that the critical variance diverges.

\begin{theorem}[Lower bound]\label{thm:critical-lower-bound}
Let \(M_N=\lfloor\log\log N\rfloor\). Then
\[
\operatorname{Var}\big(F_N(1)\big)
\geq
\frac12\sum_{k=3}^{M_N}\frac1k - o(1).
\]
In particular, there exists a constant \(C>0\) such that
\[
\operatorname{Var}\big(F_N(1)\big)
\geq
\frac12\log\log\log N - C .
\]
\end{theorem}


The logarithmic divergence in Theorem~\ref{thm:critical-lower-bound} comes
from large simple loops. We expect the same mechanism to remain valid up to the
scale \(M_N=N^{1/3}\), which corresponds to the expected critical window. The
present proof, however, only treats the much smaller scale
\(M_N=\lfloor \log\log N\rfloor\), since extending the argument to polynomially
growing \(M_N\) would require substantially sharper control of the accumulated
error terms.

\subsection{Relation with overlaps}

Following the discussion in \cite{CL}, we now recall the usual interpretation of replica symmetry in terms of
overlaps. Given two independent samples \(\sigma,\rho\) from the same
Gibbs measure, their overlap is defined by
\[
R_{\sigma \rho}=R(\sigma,\rho)
:=\frac1N\sum_{i=1}^N \sigma_i \rho_i .
\]
In the high-temperature replica-symmetric regime, two independent replicas are
asymptotically orthogonal, in the sense that
\[
\lim_{N\to\infty}\mathbb E\langle R_{\sigma \rho}^2\rangle_{\beta}=0.
\]
By contrast, in the low-temperature regime \(\beta>\beta_c\), the annealed free
energy no longer gives the correct quenched free energy, and the overlap remains
non-trivial. In particular, for the SK model one has
\[
\liminf_{N\to\infty}\mathbb E\langle R_{\sigma \rho}^2\rangle_{\beta}>0 .
\]

At the critical temperature, the precise scale of the overlap is a delicate
question. In \cite[Chapter~11]{Tal2}, Talagrand conjectures, largely motivated
by a conversation with G.~Parisi, that
\begin{equation}\label{eq:conj2}
\lim_{N\to\infty} N^{2/3}\,
\mathbb E\langle R_{\sigma \rho}^2\rangle_{1}=a
\end{equation}
exists for some constant \(a\in(0,\infty)\). We next recall and establish a
direct connection between fluctuations of the free energy and overlaps. Let
\(\langle\cdot\rangle_{\beta,t}\) denote the Gibbs measure on pairs of
configurations
\((\sigma,\rho)\in\{-1,1\}^N\times\{-1,1\}^N\) associated with the Hamiltonian
\begin{equation} \label{eq:ChaH}
\begin{split}
H_{N,t}(\sigma,\rho)
&=
\sum_{1\leq i<j\leq N}
\big(\sqrt t\, g_{ij}+\sqrt{1-t}\,g'_{ij}\big)\sigma_i\sigma_j  \\
&\qquad
+
\sum_{1\leq i<j\leq N}
\big(\sqrt t\, g_{ij}+\sqrt{1-t}\,g''_{ij}\big)\rho_i\rho_j  \\
&=
\frac12\big(\sigma,(\sqrt t\,\mathbf G+\sqrt{1-t}\,\mathbf G')\sigma\big)
+
\frac12\big(\rho,(\sqrt t\,\mathbf G+\sqrt{1-t}\,\mathbf G'')\rho\big),
\end{split}
\end{equation}
where \(\mathbf G,\mathbf G'\) and \(\mathbf G''\) are independent GOE-type
matrices with the same normalization as in \eqref{eq:HN}. Thus the two replicas
are sampled in two disorders which share a common fraction \(t\) of the
Gaussian environment. Chatterjee's variance formula gives
\begin{equation}\label{eq:Chaos}
\operatorname{Var}(F_N(\beta))
=
\frac{N\beta^2}{2}
\int_0^1
\mathbb E\big\langle R_{\sigma,\rho}^2\big\rangle_{\beta,t}\,dt
-
\frac{\beta^2}{2}.
\end{equation}
Thus fluctuations of the free energy are naturally controlled by overlaps,
although the overlap appearing in \eqref{eq:Chaos} is a disorder-chaos overlap
rather than the usual overlap between two replicas in the same environment.
Nevertheless, we prove the following comparison.

\begin{prop}\label{prop:variance-overlap}
For every \(\beta\leq 1\),
\begin{equation}\label{eq:variance-overlap}
\operatorname{Var}(F_N(\beta))
\leq
N \int_0^\beta
s \mathbb E\langle R_{\sigma \rho}^2\rangle_s
\,ds -\frac{\beta^2}{2}.
\end{equation}
\end{prop}

The proof of Proposition~\ref{prop:variance-overlap} is postponed to
Section~\ref{sec:psi-bound}. Together with Chatterjee's formula \eqref{eq:Chaos}, this shows that the size of the free-energy fluctuations is intimately related to the size of the overlap. 

\subsection{Heuristic critical picture}

We now explain how the expected critical scaling of the overlap leads to the logarithmic fluctuation scale. The following discussion is meant to explain the expected scaling. It is
motivated by Talagrand's rigorous estimates and computations in
\cite[Chapter~11.7]{Tal2}, together with the random-matrix picture developed
in \cite{BSXY}. Talagrand's results show that, outside the critical window,
\[
1-\beta^2 \gg N^{-1/3},
\]
the overlap is asymptotically given by
\[
\mathbb E\langle R_{\sigma \rho}^2\rangle_\beta
\sim
\frac{1}{N(1-\beta^2)}.
\]
At the edge of the critical window, this predicts the scale
\[
\mathbb E\langle R_{\sigma \rho}^2\rangle_\beta
\quad\text{of order}\quad
N^{-2/3}.
\]
Talagrand also proves estimates at the edge of the critical window which are consistent
with this scale. What is not available, is a precise description of the whole final window
up to the endpoint \(\beta=1\). If one nevertheless plugs the finite-size scaling prediction
\[
\mathbb E\langle R_{\sigma \rho}^2\rangle_\beta
\approx
\frac{1}{N(1-\beta^2)}
\qquad
\text{up to } 1-\beta^2 \simeq N^{-1/3},
\]
into \eqref{eq:variance-overlap}, one obtains
\[
N\int_0^{\sqrt{1-N^{-1/3}}}
s\,\frac{1}{N(1-s^2)}\,ds
=
\frac12\log\Big(\frac{1}{N^{-1/3}}\Big)
=
\frac16\log N.
\]
The remaining critical window is expected to contribute only order one, since
throughout this window one expects
\[
\mathbb E\langle R_{\sigma\rho}^2\rangle_\beta = O(N^{-2/3}).
\]
Indeed, the width of the window is of order \(N^{-1/3}\), and therefore
\[
N\int_{\sqrt{1-N^{-1/3}}}^1
s\,\mathbb E\langle R_{\sigma\rho}^2\rangle_s\,ds
=
O(1).
\]
Thus the logarithmic divergence should come from the whole high-temperature region
up to the edge of the critical window. Let us now add the random-matrix heuristic which supplies the missing endpoint picture. One has
\[
\langle R_{\sigma \rho}^2\rangle
=
\frac1{N^2}\operatorname{Tr}(M_\beta^2),
\qquad
M_\beta
=
\bigl(\langle \sigma_i\sigma_j\rangle_\beta\bigr)_{1\leq i,j\leq N}.
\]
Moreover, for fixed \(\beta<1\),
\[
M_\beta \approx (1+\beta^2-\beta G)^{-1}
\]
in a suitable Hilbert--Schmidt sense; see \cite{BSXY} for a precise statement.
If one formally extrapolates this approximation to the critical point
\(\beta=1\), then
\[
\langle R_{\sigma \rho}^2\rangle_1
\approx
\frac1{N^2}\operatorname{Tr}(2-G)^{-2}.
\]
The latter quantity is governed by the eigenvalues of \(G\) near the spectral
edge \(2\), whose fluctuations occur on the scale \(N^{-2/3}\); see, for
instance, \cite{AGZ}. This gives a random-matrix explanation for the critical
overlap scale conjectured by Talagrand and completes the heuristic picture. Let us also mention that, in the spherical SK model, the connection with random matrix theory is much more direct. Baik
and Lee \cite{BaikLee} characterized the free-energy fluctuations across the
phase transition, while Landon \cite{Landon} later analyzed the critical
regime. For Ising spins, the preceding discussion suggests that the critical
fluctuations are also governed, at least heuristically, by random-matrix edge
effects.

We finally mention a possible extension. In the SK model with non-zero external
field, one may ask whether an analogous loop mechanism appears as the
replica-symmetric stability threshold is approached. The high-temperature
expansion of spin covariances in \cite{BSV} already shows that loop
contributions play a central role away from criticality.

\section{Proof of Theorem \ref{thm:main}}

\subsection{A general convex Gaussian inequality}

The goal of this section is to state a general Gaussian inequality for the
variance. We first recall the following theorem of Chen \footnote{A related convexity mechanism already appears in the work of Auffinger and
Chen \cite{AC}, where the uniqueness of the Parisi measure is obtained through
the strict convexity of the Parisi functional, using a stochastic-control
representation.}.

\begin{theorem}[Theorem~1 in \cite{Ch1}]\label{th:chen}
Let \(G\) be a standard Gaussian vector in \(\mathbb{R}^m\), and let
\(F:\mathbb{R}^m\to\mathbb{R}\) be convex. Assume that
\[
\mathbb{E}e^{\lambda F(G)}<\infty
\qquad\text{for every }\lambda>0.
\]
Define
\[
\Psi(\lambda):=
\begin{cases}
\displaystyle
\frac{1}{\lambda}\log \mathbb{E}e^{\lambda F(G)},
& \lambda\neq 0,\\[1ex]
\mathbb{E}F(G),
& \lambda=0.
\end{cases}
\]
Then \(\Psi\) is convex on \(\mathbb{R}\).
\end{theorem}

\begin{prop}\label{prop:estvar}
Under the assumptions of Theorem~\ref{th:chen}, suppose in addition that
\[
\mathbb{E}e^{F(G)}=1.
\]
Then
\[
\operatorname{Var}\bigl(F(G)\bigr)
\leq
\log \mathbb{E}e^{2F(G)}.
\]
\end{prop}

\begin{proof}
For simplicity, write \(F=F(G)\). The Taylor expansion at the origin gives
\[
\log \mathbb{E}e^{\lambda F}
=
\lambda \mathbb{E}F
+
\frac{\lambda^2}{2}\operatorname{Var}(F)
+
o(\lambda^2).
\]
Therefore,
\[
\Psi(\lambda)
=
\mathbb{E}F
+
\frac{\lambda}{2}\operatorname{Var}(F)
+
o(\lambda),
\]
and hence
\[
\Psi'(0)=\frac12\operatorname{Var}(F).
\]
Moreover, the normalization \(\mathbb{E}e^F=1\) gives
\[
\Psi(1)=\log \mathbb{E}e^F=0.
\]
Since \(\Psi\) is convex, its derivative at \(0\) is bounded above by the
slope of the secant line between \(1\) and \(2\). Thus
\[
\Psi'(0)
\leq
\frac{\Psi(2)-\Psi(1)}{2-1}
=
\Psi(2).
\]
Consequently,
\[
\frac12\operatorname{Var}(F)
\leq
\Psi(2)
=
\frac12\log \mathbb{E}e^{2F},
\]
which proves the claim.
\end{proof}

\subsection{Application to the SK model}\label{SubSect:AppSK}

\begin{proof} [Proof of Theorem~\ref{thm:main}] 

We first verify that the assumptions of Theorem~\ref{th:chen} are satisfied.
In order to apply the Gaussian convexity theorem, we regard the free energy as
a function of a standard Gaussian vector. Thus we write
\[
g_{ij}=\frac1{\sqrt N}g'_{ij},
\qquad 1\leq i<j\leq N,
\]
where the variables \(g'_{ij}\) are independent standard Gaussian random
variables. Equivalently,
\[
H_N(\sigma)
=
\frac1{\sqrt N}
\sum_{1\leq i<j\leq N} g'_{ij}\sigma_i\sigma_j .
\]
We now check that \(F_N(\beta)\) is convex as a function of the variables
\((g'_{ij})_{i<j}\). For \(1\leq i<j\leq N\),
\[
\partial_{g'_{ij}}F_N(\beta)
=
\frac{\beta}{\sqrt N}
\bigl\langle \sigma_i\sigma_j\bigr\rangle_\beta .
\]
Therefore,
\[
\partial_{g'_{ij}}\partial_{g'_{kl}}F_N(\beta)
=
\frac{\beta^2}{N}
\left(
\bigl\langle \sigma_i\sigma_j\sigma_k\sigma_l\bigr\rangle_\beta
-
\bigl\langle \sigma_i\sigma_j\bigr\rangle_\beta
\bigl\langle \sigma_k\sigma_l\bigr\rangle_\beta
\right).
\]
Equivalently, for any real coefficients \((a_{ij})_{i<j}\),
\[
\sum_{i<j}\sum_{k<l}
a_{ij}a_{kl}
\partial_{g'_{ij}}\partial_{g'_{kl}}F_N(\beta)
=
\frac{\beta^2}{N}
\left\langle
\left(
\sum_{i<j}a_{ij}
\Bigl(
\sigma_i\sigma_j
-
\bigl\langle \sigma_i\sigma_j\bigr\rangle_\beta
\Bigr)
\right)^2
\right\rangle_\beta
\geq 0 .
\]
Hence the Hessian of \(F_N(\beta)\), viewed as a function of the standard Gaussian disorder \(g'\), is positive semidefinite. Thus \(F_N(\beta)\) is convex. It remains to check exponential integrability. Since
\[
F_N(\beta)
\leq
\beta \max_{\sigma}H_N(\sigma),
\]
for every \(\lambda>0\),
\[
e^{\lambda F_N(\beta)}
\leq
e^{\lambda \beta \max_{\sigma}H_N(\sigma)}
\leq
\sum_{\sigma}e^{\lambda \beta H_N(\sigma)}.
\]
Therefore,
\[
\mathbb{E}e^{\lambda F_N(\beta)}
\leq
\sum_{\sigma}
\mathbb{E}e^{\lambda \beta H_N(\sigma)}
<\infty,
\]
because, for every fixed \(\sigma\), \(H_N(\sigma)\) is Gaussian. Define
\[
W_N(\beta)
:=
\frac{Z_N(\beta)}{\mathbb E Z_N(\beta)}.
\]
Then
\[
\log W_N(\beta)
=
F_N(\beta)-\log \mathbb E Z_N(\beta).
\]
Since \(\log \mathbb E Z_N(\beta)\) is deterministic and \(F_N(\beta)\) is convex as a
function of the Gaussian disorder, \(\log W_N(\beta)\) is also convex. Moreover, for every \(\lambda>0\),
\[
\mathbb E e^{\lambda\log W_N(\beta)}=\mathbb E W_N(\beta)^\lambda<\infty.
\]
Hence all the assumptions of Theorem~\ref{th:chen} are satisfied, and
\[
\Psi_N(\lambda)
:=
\begin{cases}
\displaystyle
\frac1\lambda\log\mathbb{E}e^{\lambda \log W_N(\beta)},
& \lambda\neq0,\\[1ex]
\mathbb{E} \log W_N(\beta),
& \lambda=0,
\end{cases}
\]
is convex on \(\mathbb{R}\). Since $\bE W_N(\beta)=1$, the proposition \ref{prop:estvar} yields
\[
\Var(\log Z_N(\beta))
=
\Var(\log W_N(\beta))
\leq
\log\frac{\bE Z_N(\beta)^2}{(\bE Z_N(\beta))^2}.
\]
We now compute the right-hand side. For two independent spin configurations
$\sigma,\tau$, Gaussian integration gives
\[
\frac{\bE Z_N(\beta)^2}{(\bE Z_N(\beta))^2}
=
\bE_{\sigma,\tau}
\exp\left(
\frac{\beta^2}{N}
\sum_{i<j}\sigma_i\sigma_j\tau_i\tau_j
\right).
\]
Writing $s_i=\sigma_i\tau_i$, the variables $s_i$ are independent Rademacher variables and
\[
\sum_{i<j}s_is_j
=
\frac12\left[\left(\sum_{i=1}^Ns_i\right)^2-N\right].
\]
Thus
\[
\frac{\bE Z_N(\beta)^2}{(\bE Z_N(\beta))^2}
=
e^{-\beta^2/2}
\bE_s\exp\left(
\frac{\beta^2}{2N}
\left(\sum_{i=1}^Ns_i\right)^2
\right).
\]
The right-hand side is, up to the deterministic factor $e^{-\beta^2/2}$,
the Curie--Weiss partition function at inverse temperature $\beta^2$.

\begin{lemma}\label{lem:critical-CW}
Let \(s_1,\ldots,s_N\) be independent Rademacher random variables. Then
\[
\bE_s\exp\left(
\frac{1}{2N}\left(\sum_{i=1}^N s_i\right)^2
\right)
\sim
\frac{N^{1/4}}{\sqrt{2\pi}}
\int_{\mathbb R}e^{-y^4/12}\,dy.
\]
In particular,
\[
\bE_s\exp\left(
\frac{1}{2N}\left(\sum_{i=1}^N s_i\right)^2
\right)
\leq C N^{1/4}.
\]
\end{lemma}

\begin{proof}
By the Hubbard--Stratonovich identity,
\[
e^{u^2/2}
=
\frac{1}{\sqrt{2\pi}}
\int_{\mathbb R}e^{-z^2/2+uz}\,dz,
\]
we have
\[
\begin{aligned}
\bE_s\exp\left(
\frac{1}{2N}\left(\sum_{i=1}^N s_i\right)^2
\right)
&=
\frac{1}{\sqrt{2\pi}}
\int_{\mathbb R}
e^{-z^2/2}
\bE_s\exp\left(
\frac{z}{\sqrt N}\sum_{i=1}^N s_i
\right)\,dz\\
&=
\frac{1}{\sqrt{2\pi}}
\int_{\mathbb R}
e^{-z^2/2}
\cosh^N\left(\frac{z}{\sqrt N}\right)\,dz.
\end{aligned}
\]
After the change of variables \(z=\sqrt N\,x\), this becomes
\[
\sqrt{\frac{N}{2\pi}}
\int_{\mathbb R}e^{N\phi(x)}\,dx,
\qquad
\phi(x):=\log\cosh x-\frac{x^2}{2}.
\]

We first analyze the contribution near the maximum of \(\phi\). Since
\[
\phi(x)
=
-\frac{x^4}{12}+O(x^6)
\qquad \text{as } x\to 0,
\]
we may choose \(\delta>0\) small enough and \(c>0\) such that
\[
\phi(x)\leq -c x^4
\qquad \text{for all } |x|\leq \delta.
\]
After the change of variables \(x=N^{-1/4}y\), we obtain
\[
N^{1/4}\int_{|x|\leq \delta}e^{N\phi(x)}\,dx
=
\int_{\mathbb R}
e^{N\phi(N^{-1/4}y)}
\mathbf 1_{\{|y|\leq \delta N^{1/4}\}}\,dy.
\]
For every fixed \(y\),
\[
N\phi(N^{-1/4}y)
=
-\frac{y^4}{12}+O(N^{-1/2}y^6)
\longrightarrow
-\frac{y^4}{12},
\]
and, for \(|y|\leq \delta N^{1/4}\),
\[
e^{N\phi(N^{-1/4}y)}
\leq e^{-cy^4}.
\]
Since \(e^{-cy^4}\in L^1(\mathbb R)\), dominated convergence gives
\[
N^{1/4}\int_{|x|\leq \delta}e^{N\phi(x)}\,dx
\longrightarrow
\int_{\mathbb R}e^{-y^4/12}\,dy.
\]

It remains to show that the contribution away from \(0\) is negligible.
We record the elementary properties of \(\phi\). Since
\[
\phi''(x)= -\tanh^2(x)\leq 0,
\]
the function \(\phi\) is concave. Moreover,
\[
\phi'(x)=\tanh x-x,
\]
so \(0\) is the unique critical point, hence the unique global maximizer,
and \(\phi(0)=0\). Therefore, for every fixed \(M>\delta\),
\[
\sup_{\delta\leq |x|\leq M}\phi(x)<0.
\]
Thus
\[
\int_{\delta\leq |x|\leq M}e^{N\phi(x)}\,dx
\]
is exponentially small in \(N\). Finally, using \(\log\cosh x\leq |x|\), we get
\[
\phi(x)\leq |x|-\frac{x^2}{2}.
\]
Choosing \(M\) large enough, this implies
\[
\phi(x)\leq -\frac{x^2}{4}
\qquad \text{for } |x|\geq M.
\]
Hence
\[
\int_{|x|\geq M}e^{N\phi(x)}\,dx
\leq
\int_{|x|\geq M}e^{-Nx^2/4}\,dx,
\]
which is exponentially small in \(N\). Therefore
\[
N^{1/4} \int_{\mathbb R}e^{N\phi(x)}\,dx
\sim
N^{1/4}  \int_{|x|\leq\delta}e^{N\phi(x)}\,dx
\sim
   \int_{\mathbb R}e^{-y^4/12}\,dy.
\]
Multiplying by $N^{\frac{1}{4}} / \sqrt{2\pi}$, we obtain
\[
\bE_s\exp\left(
\frac{1}{2N}\left(\sum_{i=1}^N s_i\right)^2
\right)
\sim
\frac{N^{1/4}}{\sqrt{2\pi}}
\int_{\mathbb R}e^{-y^4/12}\,dy.
\]
The upper bound by \(C N^{1/4}\) follows immediately from the asymptotic.
\end{proof}

At the critical value \(\beta=1\), Lemma~\ref{lem:critical-CW} gives
\[
\bE W_N(1)^2
\sim
\frac{N^{1/4} e^{-1/2}}{\sqrt{2\pi}} \int_{\mathbb R}e^{-y^4/12}\,dy.
\]
In particular,
\[
\log \bE W_N(1)^2
=
\frac14\log N+O(1).
\]
Since
\[
\Var(\log Z_N(1))
=
\Var(\log W_N(1))
\leq
\log \bE W_N(1)^2,
\]
we obtain
\[
\Var(\log Z_N(1))
\leq
\frac14\log N+O(1).
\]

\end{proof}

\section{Proof of Theorem \ref{thm:critical-lower-bound}}

\begin{proof}[Proof of Theorem~\ref{thm:critical-lower-bound}]
We use Chatterjee's variance formula \cite[Theorem~3.11]{Cha}. If
\(g'=(g_1',\ldots,g_n')\) is a vector of independent standard Gaussian random
variables and \(f\) is a \(C^\infty\) function of \(g'\) with bounded derivatives
of all orders, then
\[
\operatorname{Var}(f(g'))
=
\sum_{k=1}^{\infty}
\frac1{k!}
\sum_{1\leq i_1,\ldots,i_k\leq n}
\left(
\mathbb E
\frac{\partial^k f}
{\partial g_{i_1}'\cdots \partial g_{i_k}'}(g')
\right)^2 .
\]

We apply this formula to the standard Gaussian variables
\[
g'=(g'_{ij})_{1\le i<j\le N},
\qquad
g_{ij}=\frac{1}{\sqrt N}g'_{ij}.
\]
Thus
\[
H_N(\sigma)
=
\frac1{\sqrt N}
\sum_{1\le i<j\le N}
g'_{ij}\sigma_i\sigma_j .
\]

Let \(\mathcal C_{N,k}\) be the set of unoriented simple cycles of length \(k\)
in the complete graph on \(N\) vertices. For \(\gamma\in\mathcal C_{N,k}\), let
\(E(\gamma)\) be its set of edges and define
\[
\partial_\gamma^k
=
\prod_{e\in E(\gamma)}\partial_{g'_e}.
\]
Since Chatterjee's formula sums over ordered \(k\)-tuples of Gaussian
coordinates, each fixed \(\gamma\in\mathcal C_{N,k}\) gives \(k!\) identical
terms, corresponding to the \(k!\) orderings of its edges. Therefore, for any
integer \(M_N\),
\[
\operatorname{Var}\bigl(\log Z_N(1)\bigr)
\geq
\sum_{k=3}^{M_N}
\sum_{\gamma\in\mathcal C_{N,k}}
\left(
\mathbb E \partial_\gamma^k \log Z_N(1)
\right)^2 .
\]
We set $\langle\cdot\rangle_1=\langle\cdot\rangle$. We now identify the contribution of a fixed simple cycle. Write
\(\gamma=(i_1,\ldots,i_k)\), with \(i_{k+1}=i_1\), and put
\[
X_\ell=\sigma_{i_\ell}\sigma_{i_{\ell+1}},
\qquad 1\leq \ell\leq k.
\]
Repeated differentiation of the free energy gives
\[
\partial_\gamma^k \log Z_N(1)
=
N^{-k/2}
\kappa_{\langle\cdot\rangle}(X_1,\ldots,X_k),
\]
where \(\kappa_{\langle\cdot\rangle}\) denotes the Gibbs cumulant. Thus
\[
\kappa_{\langle\cdot\rangle}(X_1,\ldots,X_k)
=
\sum_{\pi\in\mathcal P_k}
(|\pi|-1)!(-1)^{|\pi|-1}
\prod_{B\in\pi}
\left\langle \prod_{\ell\in B}X_\ell\right\rangle ,
\]
where \(\mathcal P_k\) is the set of partitions of \(\{1,\ldots,k\}\), and
\(|\pi|\) denotes the number of blocks of \(\pi\). The one-block partition gives
\[
\left\langle \prod_{\ell=1}^k X_\ell\right\rangle
=
\left\langle
\prod_{\ell=1}^k \sigma_{i_\ell}\sigma_{i_{\ell+1}}
\right\rangle
=
\langle 1\rangle
=
1.
\]
All other partitions have at least two blocks. Hence every block
\(B\in\pi\) is a non-empty strict subset of \(\{1,\ldots,k\}\). Consider
\[
\prod_{\ell\in B}X_\ell
=
\prod_{\ell\in B}\sigma_{i_\ell}\sigma_{i_{\ell+1}}.
\]
A spin \(\sigma_{i_j}\) cancels from this product only if the two edges of the
cycle that contain \(i_j\) both belong to \(B\). Since \(B\) is a non-empty
proper subset of \(\{1,\ldots,k\}\), there is at least one place along the cycle
where an edge in \(B\) is followed by an edge outside \(B\). At this vertex, the
associated spin appears once in the product. Thus the product is a non-trivial
spin monomial, and in particular it is not identically equal to \(1\).
We shall use the following consequence of Talagrand's critical overlap
estimate \cite[Proposition 11.7.8]{Tal2}: at the critical point \(\beta=1\), there exists \(C<\infty\) such that
\[
\mathbb E\langle R_{12}^2\rangle\leq C N^{-1/2}.
\]

\begin{lemma}\label{lem:local-correlations}
Let \(A\subset\{1,\ldots,N\}\) be non-empty, with \(|A|\leq \sqrt N\). At zero
external field and at \(\beta=1\),
\[
\mathbb E\langle \sigma_A\rangle^2\leq C N^{-1/2},
\]
provided \(|A|\geq 2\). If \(|A|\) is odd, then in fact
\(\langle\sigma_A\rangle=0\).
\end{lemma}

\begin{proof}
If \(|A|\) is odd, the claim follows from the global spin-flip symmetry. Assume
therefore that \(r:=|A|\) is even. Let
\[
\tau_i=\sigma_i^1\sigma_i^2,
\qquad
R_{12}=\frac1N\sum_{i=1}^N\tau_i.
\]
Then
\[
\mathbb E\langle \sigma_A\rangle^2
=
\mathbb E\left\langle \prod_{i\in A}\tau_i\right\rangle .
\]
By exchangeability, all subsets of size \(r\) give the same contribution. In the
expansion of \(\mathbb E\langle R_{12}^r\rangle\), the contribution of the
ordered \(r\)-tuples of distinct indices is therefore
\[
\frac{N(N-1)\cdots(N-r+1)}{N^r}
\mathbb E\langle \sigma_A\rangle^2.
\]
Moreover, since \(r\) is even, every term in the expansion of
\(\mathbb E\langle R_{12}^r\rangle\) is non-negative: after cancelling repeated
indices, it is of the form
\(\mathbb E\langle \sigma_{A'}\rangle^2\) for some set \(A'\). Hence
\[
\frac{N(N-1)\cdots(N-r+1)}{N^r}
\mathbb E\langle \sigma_A\rangle^2
\leq
\mathbb E\langle R_{12}^r\rangle .
\]
Since \(|R_{12}|\leq1\) and \(r\geq2\),
\[
\mathbb E\langle R_{12}^r\rangle
\leq
\mathbb E\langle R_{12}^2\rangle
\leq
C N^{-1/2}.
\]
Finally, since \(r\leq \sqrt N\), and for \(N\) large \(r\leq N/2\),
\[
\frac{N^r}{N(N-1)\cdots(N-r+1)}
\leq
\left(\frac{N}{N-r}\right)^r
=
\exp\left(-r\log\left(1-\frac rN\right)\right)
\leq
\exp\left(\frac{r^2}{N-r}\right)
\leq C.
\]
This proves the lemma.
\end{proof}

We return to the cumulant expansion. Using Lemma~\ref{lem:local-correlations}
and the fact that all Gibbs averages are bounded by one in absolute value, we
obtain
\begin{align}\label{eq:maj}
\mathbb E \left|
\sum_{\substack{\pi\in\mathcal P_k\\ |\pi|\ge 2}}
(|\pi|-1)!(-1)^{|\pi|-1}
\prod_{B\in\pi}
\left\langle \prod_{\ell\in B} X_\ell\right\rangle
\right|
&\leq
\sum_{\substack{\pi\in\mathcal P_k\\ |\pi|\ge 2}}
(|\pi|-1)!
\mathbb E
\prod_{B\in\pi}
\left|
\left\langle \prod_{\ell\in B} X_\ell\right\rangle
\right|  \\
&\leq
C N^{-1/4}
\sum_{\substack{\pi\in\mathcal P_k\\ |\pi|\ge 2}}
(|\pi|-1)! .
\end{align}
Indeed, if \(|\pi|\geq2\), then \(\pi\) contains a proper block
\(B\subsetneq\{1,\ldots,k\}\). For this block,
\[
\prod_{\ell\in B}X_\ell=\sigma_A
\]
for some non-empty set \(A\), and by Cauchy--Schwarz and
Lemma~\ref{lem:local-correlations},
\[
\mathbb E\left|
\left\langle \prod_{\ell\in B}X_\ell\right\rangle
\right|
\leq
\left(
\mathbb E
\left\langle \prod_{\ell\in B}X_\ell\right\rangle^2
\right)^{1/2}
\leq
C N^{-1/4}.
\]
All the other factors are bounded by one in absolute value.

It remains to bound the combinatorial factor. For fixed \(m\), the number of
partitions of \(\{1,\ldots,k\}\) with \(m\) blocks is bounded by \(m^k\):
indeed, after labelling the blocks, each of the \(k\) indices has at most \(m\)
possible choices for the block to which it belongs. Therefore,
\[
\sum_{\substack{\pi\in\mathcal P_k\\ |\pi|\ge 2}}
(|\pi|-1)!
\leq
\sum_{m=2}^k m^k(m-1)!
\leq
k^{k+1}k!.
\]
Combining this with \eqref{eq:maj}, we get
\[
\mathbb E \left|
\sum_{\substack{\pi\in\mathcal P_k\\ |\pi|\ge 2}}
(|\pi|-1)!(-1)^{|\pi|-1}
\prod_{B\in\pi}
\left\langle \prod_{\ell\in B} X_\ell\right\rangle
\right|
\leq
C k^{k+1}k! N^{-1/4}.
\]
Hence, for every simple cycle \(\gamma\in\mathcal C_{N,k}\),
\[
\mathbb E\,\partial_\gamma^k \log Z_N(1)
=
N^{-k/2}\left(1+\varepsilon_{N,k,\gamma}\right),
\]
where
\[
|\varepsilon_{N,k,\gamma}|
\leq
C k^{k+1}k! N^{-1/4}.
\]

We now sum over cycles. Since \(M_N=\lfloor \log\log N\rfloor\), we have
\[
\sup_{3\leq k\leq M_N} k^{k+1}k!N^{-1/4}=o(1).
\]
Therefore,
\begin{align*}
\sum_{k=3}^{M_N}
\sum_{\gamma\in\mathcal C_{N,k}}
\left(
\mathbb E\,\partial_\gamma^k \log Z_N(1)
\right)^2
&=
\sum_{k=3}^{M_N}
N^{-k}
\sum_{\gamma\in\mathcal C_{N,k}}
\left(1+\varepsilon_{N,k,\gamma}\right)^2 \\
&=
\sum_{k=3}^{M_N}
N^{-k}|\mathcal C_{N,k}|+o(1).
\end{align*}
Here we used that \(N^{-k}|\mathcal C_{N,k}|\leq 1\) and the error is uniform
for \(k\leq M_N\).

Finally,
\[
|\mathcal C_{N,k}|
=
\frac{N(N-1)\cdots(N-k+1)}{2k}.
\]
Thus, uniformly for \(k\leq M_N\),
\[
N^{-k}|\mathcal C_{N,k}|
=
\frac1{2k}\prod_{j=0}^{k-1}\left(1-\frac jN\right)
=
\frac1{2k}+o\left(\frac1k\right),
\]
because \(M_N^2/N\to0\). Hence
\[
\sum_{k=3}^{M_N}
N^{-k}|\mathcal C_{N,k}|
=
\frac12\sum_{k=3}^{M_N}\frac1k+o(1).
\]
Combining the previous estimates gives
\[
\operatorname{Var}\big(\log Z_N(1)\big)
\geq
\frac12\sum_{k=3}^{M_N}\frac1k+o(1).
\]
Since
\[
\sum_{k=3}^{M_N}\frac1k
=
\sum_{k=1}^{M_N}\frac1k-\frac32,
\]
this is equivalent, up to an additive constant, to the announced logarithmic
divergence. In particular, for \(M_N=\lfloor\log\log N\rfloor\),
\[
\operatorname{Var}\big(\log Z_N(1)\big)
\geq
\frac12\log\log\log N-C.
\]
\end{proof}

\section{A variance bound from the annealed--quenched gap}
\label{sec:psi-bound}

\begin{proof}[Proof of Proposition~\ref{prop:variance-overlap}]
We argue as in the proof of Proposition~\ref{prop:estvar}, but we use the
secant line between \(0\) and \(1\), rather than the one between \(1\) and \(2\).
Let \(F\) satisfy the assumptions of Theorem~\ref{th:chen}, and assume in
addition that
\[
\mathbb E e^{F(G)}=1 .
\]
Recall that
\[
\Psi(\lambda)
=
\frac1\lambda \log \mathbb E e^{\lambda F(G)}
\]
is convex, with continuous extension
\[
\Psi(0)=\mathbb E F(G).
\]
Moreover,
\[
\Psi'(0)
=
\frac12\operatorname{Var}(F(G)).
\]
By convexity,
\[
\Psi'(0)
\leq
\Psi(1)-\Psi(0).
\]
Since \(\mathbb E e^{F(G)}=1\), we have \(\Psi(1)=0\), and therefore
\[
\frac12\operatorname{Var}(F(G))
\leq
-\Psi(0)
=
-\mathbb E F(G).
\]

We apply this inequality to the centered free energy
\[
F(G)
=
F_N(\beta)-\log\mathbb E Z_N(\beta).
\]
The assumptions of Theorem~\ref{th:chen} are satisfied for the SK model, as
explained in Section~\ref{SubSect:AppSK}. We obtain
\[
\frac12\operatorname{Var}(F_N(\beta))
\leq
\log\mathbb E Z_N(\beta)-\mathbb E F_N(\beta).
\]
Thus
\begin{equation}\label{eq:var-aq-gap}
\operatorname{Var}(F_N(\beta))
\leq
2\left(
\log\mathbb E Z_N(\beta)-\mathbb E F_N(\beta)
\right).
\end{equation}

It remains to express the annealed--quenched gap in terms of the usual overlap.
By Gaussian integration by parts,
\[
\frac{d}{d\beta}\mathbb E F_N(\beta)
=
\mathbb E\langle H_N(\sigma)\rangle_\beta
=
\frac{\beta N}{2}
\left(
1-\mathbb E\langle R_{12}^2\rangle_\beta
\right).
\]
On the other hand, since
\[
\log\mathbb E Z_N(\beta)
=
\frac{\beta^2}{4}(N-1),
\]
we have
\[
\frac{d}{d\beta}
\left(
\log\mathbb E Z_N(\beta)-\mathbb E F_N(\beta)
\right)
=
\frac{\beta N}{2}
\left(
\mathbb E\langle R_{12}^2\rangle_\beta-\frac1N
\right).
\]
Since the annealed--quenched gap vanishes at \(\beta=0\), integration gives
\[
\log\mathbb E Z_N(\beta)-\mathbb E F_N(\beta)
=
\frac N2
\int_0^\beta
s\left(
\mathbb E\langle R_{12}^2\rangle_s-\frac1N
\right)\,ds .
\]
Combining this identity with \eqref{eq:var-aq-gap}, we obtain
\[
\operatorname{Var}(F_N(\beta))
\leq
N\int_0^\beta
s\left(
\mathbb E\langle R_{12}^2\rangle_s-\frac1N
\right)\,ds,
\]
which proves Proposition~\ref{prop:variance-overlap}.
\end{proof}

\vspace{0.2cm}
\noindent\textbf{Acknowledgements.}
The author is grateful to Christian Brennecke for his encouragement, support,
and many helpful discussions related to this line of research. He also thanks
Loucas Pillaud-Vivien, Nicola Kistler, and Ella Hiesmayr for stimulating
conversations, and Jean-Christophe Mourrat for helpful discussions and valuable
comments on an earlier version of this note. The author was supported by the ERC
MSCA grant SLOHD (101203974).


\end{document}